\documentclass[12pt,a4paper]{amsart}
\usepackage{amsmath,amssymb,color,multicol,setspace}
\usepackage[utf8,utf8x]{inputenc}

\usepackage[pdftitle = {Frieze\ patterns\ over\ algebraic\ numbers}]{hyperref}
\usepackage{longtable}
\usepackage{xy,amscd}
\usepackage{epsfig}
\usepackage{rotating}
\usepackage{xspace}
\usepackage{enumitem}
\usepackage{fullpage}
\usepackage{changes}

\newtheorem{Lemma}{Lemma}[section]
\newtheorem{Theorem}[Lemma]{Theorem}
\newtheorem{Proposition}[Lemma]{Proposition}
\newtheorem{Corollary}[Lemma]{Corollary}

\theoremstyle{definition}
\newtheorem{Definition}[Lemma]{Definition}

\newtheorem{Conjecture}[Lemma]{Conjecture}

\newtheorem{Remark}[Lemma]{Remark}

\newtheorem{Example}[Lemma]{Example}

\numberwithin{equation}{section}

\DeclareMathOperator{\quot}{Quot}
\newcommand{\cR}{R^\circ}
\DeclareMathOperator{\md}{mod}
\DeclareMathOperator{\End}{End}

\title{Frieze patterns over algebraic numbers}

\author{Michael~Cuntz}
\address{Michael Cuntz, Leibniz Universit\"at Hannover,
Institut f\"ur Algebra, Zahlentheorie und Diskrete Mathematik,
Fakult\"at f\"ur Mathematik und Physik,
Welfengarten 1,
D-30167 Hannover, Germany}
\email{cuntz@math.uni-hannover.de}
\urladdr{https://www.iazd.uni-hannover.de/de/cuntz}

\author{Thorsten~Holm}
\address{Thorsten Holm, Leibniz Universit\"at Hannover,
Institut f\"ur Algebra, Zahlentheorie und Diskrete Mathematik,
Fakult\"at f\"ur Mathematik und Physik,
Welfengarten 1,
D-30167 Hannover, Germany}
\email{holm@math.uni-hannover.de}
\urladdr{https://www.iazd.uni-hannover.de/de/holm}

\author{Carlo~Pagano}
\address{Carlo Pagano, Concordia University, Department of Mathematics and Statistics, Montreal, Quebec H3G 1M8, Canada}
\email{carlo.pagano@concordia.ca}
\urladdr{https://sites.google.com/view/carlopagano}



\begin{document}

\begin{abstract}
Conway and Coxeter have shown that frieze patterns over positive rational integers are in bijection with triangulations of polygons. An investigation of frieze patterns over other subsets of the complex numbers has recently been initiated by J{\o}rgensen and the first two authors. In this paper we first show that a ring of algebraic numbers has finitely many units if and only if it is an order in a quadratic number field $\mathbb{Q}(\sqrt{d})$ where $d<0$. We conclude that these are exactly the rings of algebraic numbers over which there are finitely many non-zero frieze patterns for any given height. We then show that apart from the cases $d\in \{-1,-2,-3,-7,-11\}$ all non-zero frieze patterns over the rings of integers $\mathcal{O}_d$ for $d<0$ have only integral entries and hence are known as (twisted) Conway-Coxeter frieze patterns.
\end{abstract}

\maketitle

\section{Introduction}

Frieze patterns are arrays of numbers introduced by Coxeter \cite{Cox71} (for the definition see Section \ref{sec:algnum}).
They are closely connected to Fomin and Zelevinsky's cluster algebras of Dynkin type $A$.
This connection to cluster algebras is one of the main reasons for frieze patterns being an active topic of research, linking different areas like combinatorics, geometry, and representation theory; see the survey \cite{MG15}.  

Soon after Coxeter defined frieze patterns, Conway and Coxeter \cite{CC73} gave a beautiful characterization of frieze patterns over the positive rational integers: such frieze patterns of height $n$ are in bijection with triangulations of a regular $(n+3)$-gon. 
It is an obvious question whether this observation of Conway and Coxeter could be generalized from triangulations to other dissections of polygons. About fourty years later, several variations were proposed.

In \cite{CH17} we give a combinatorial model for frieze patterns over the rational integers (including negative numbers and $0$) and suggest to consider further subsets of the complex numbers. This extends the case of frieze patterns over $\mathbb{Z}\setminus\{0\}$ which was also considered in \cite{F14}.
Our combinatorial model consists of triangles and quadrangles plus a combinatorial recipe to determine the entries of the corresponding frieze patterns from the dissection.

In \cite{HJ17} the authors define for any integer $p\ge 3$ the notion of a frieze pattern of type $\Lambda_p$: a frieze pattern (with positive real numbers as entries) is of type $\Lambda_p$ if the quiddity sequence consists of (positive) integral multiples of the number $\lambda_p=2\cos(\frac{\pi}{p})$. 
Note that for $p=3$ we have $\lambda_3= 1$ and frieze patterns of type $\Lambda_3$ are precisely the Conway-Coxeter frieze patterns (over positive rational integers). Then the classic result by Conway and Coxter can be generalized (see \cite[Theorem A]{HJ17}):
{\em There is a bijection between $p$-angulations of the $(n+3)$-gon and frieze patterns of type $\Lambda_p$ with height $n$.}
One can even go a bit further: to any dissection of a polygon one can associate a frieze pattern with entries in $\mathcal{O}_K$, the ring of algebraic integers of the number field $K=\mathbb{Q}(\lambda_ {p_1},\ldots, \lambda_{p_s})$, where $p_1,\ldots,p_s$ are the sizes of the subpolygons in the dissection. This yields an injection from polygon dissections of the $(n+3)$-gon to frieze patterns with height $n$ (see \cite[Theorem B]{HJ17}).

In \cite{Cuntz-comb}, a notion of irreducible frieze pattern was introduced. Using this notion one obtains combinatorial models for the set of frieze patterns with entries in arbitrary subsets of a commutative ring. The set of irreducible frieze patterns for a given subset is infinite in most cases, but it turns out that all the previously known combinatorial models arise in this way.
\medskip

One of the most fundamental questions in the theory of frieze patterns is whether over a given set of numbers there are finitely or infinitely many non-zero frieze patterns for any given height. Usually, one can only expect a nice combinatorial model when the number of frieze patterns in each height is finite.
In \cite{CH17} the following results are shown:
\begin{enumerate}
\item If $R\subseteq\mathbb{C}$ is a discrete subset, then there are only finitely many non-zero frieze patterns over $R$ of height $n$ for each $n\in\mathbb{N}$ (see
\cite[Corollary 3.8]{CH17}).
\item Let $R\subseteq\mathbb{C}$ be a subset containing infinitely many divisors of $2$. Then for each $n>0$ there are infinitely many non-zero frieze patterns of height $n$ over $R$ (see \cite[Proposition 3.9]{CH17}).
\end{enumerate}

It is known that a subring $R$ of the complex numbers forms a discrete subset if and only if $R$ is contained in the ring of integers $\mathcal{O}_d$ of an imaginary quadratic number field $\mathbb{Q}(\sqrt{d})$ for $d<0$ (see Proposition \ref{discreteQd}).
So according to (1) above these rings of integers will lead to finitely many non-zero frieze patterns in each height.
On the other hand, any ring of integers $\mathcal{O}_d$ of a real quadratic field $\mathbb{Q}(\sqrt{d})$ with $d>0$ contains infinitely many units by Dirichlet's unit theorem. Hence by (2) above, there are infinitely many non-zero frieze patterns in each positive height over $\mathcal{O}_d$ for any $d>0$.

One of the aims of this paper is to deduce for many more subrings $R$ of the complex numbers whether or not there are finitely or infinitely many non-zero frieze patterns.
It turns out that the situation becomes very different as soon as transcendental numbers appear in a frieze pattern, therefore we concentrate on subrings of the ring $\overline{\mathbb{Q}}$ of algebraic numbers.

Our first main result shows that having finitely many non-zero frieze patterns in each height restricts the possible subrings.
This theorem will be restated and proven as Theorem~\ref{thm:Rcircv2} below.

\begin{Theorem} \label{thm:Rcirc}
Let $R\le \mathbb{C}$ be a subring and let $\cR$ be the subring of $R$ generated by all entries of all non-zero frieze patterns over $R$.
If $\cR\subseteq \overline{\mathbb{Q}}$, then there are finitely many non-zero frieze patterns over $R$ in each positive height if and only if
\begin{enumerate}
 \item[{(i)}] $\cR=\mathbb{Z}$\quad or
 \item[{(ii)}] $\cR$ is an order in $\mathbb{Q}(\sqrt{d})$ with $d\in \{-1,-2,-3,-7,-11\}$.
\end{enumerate}
\end{Theorem}

\noindent
In order to prove Theorem \ref{thm:Rcirc}, we need the following purely number theoretic result on the number of units in rings of algebraic numbers:
\begin{Theorem}[Theorem \ref{thmunits}]
Let $R\le \overline{\mathbb{Q}}$ be a subring with finitely many units. Then $R=\mathbb{Z}$ or there exists an integer $d\in\mathbb{Z}_{<0}$ such that $R$ is an order in $\mathbb{Q}(\sqrt{d})$.
\end{Theorem}

Theorem \ref{thm:Rcirc} shows that the imaginary quadratic fields and their rings of integers play a special role when studying non-zero frieze patterns over arbitrary number fields.

Clearly, every frieze pattern over the rational integers $\mathbb{Z}$ is also a frieze pattern over $R$ for any subring $R\le \mathbb{C}$. 
The non-zero frieze patterns over $\mathbb{Z}$ are known. First of all there are the frieze patterns over positive rational integers considered by Conway and Coxeter. 
When also allowing negative integers as entries, not very many further frieze patterns occur (see \cite{F14}): every non-zero frieze pattern over $\mathbb{Z}$ is either a Conway-Coxeter frieze pattern or can be obtained from a Conway-Coxeter frieze pattern by multiplying every second diagonal by $-1$ (which only works if the height $n$ is odd). We call the latter {\em twisted Conway-Coxeter frieze patterns}.
In Section \ref{sec:reduction} we give a self-contained proof of this fact, using different methods than in \cite{F14}.
In particular, there are only finitely many non-zero frieze patterns over the rational integers.

For understanding frieze patterns over rings of integers in quadratic number fields the fundamental question is how many new such frieze patterns appear in addition to the non-zero frieze patterns over $\mathbb{Z}$ already known. As our second main result we can answer this question for almost all rings of integers in imaginary quadratic number fields.
This theorem will be restated and proven as Theorem \ref{thm:imaginary} below.

\begin{Theorem}\label{thmint_3}
Let $d$ be a negative square-free integer. Then
\[ \mathcal{O}_d^{\circ} = \begin{cases}
 \mathcal{O}_d & \text{if } d\in \{-1,-2,-3,-7,-11\}, \\
 \mathbb{Z} & \text{else}.
\end{cases} 
\]
In other words,
for any negative square-free integer $d\not\in \{-1,-2,-3,-7,-11\}$ 
the only non-zero frieze patterns over $\mathcal{O}_d$ are the Conway-Coxeter frieze patterns and the related twisted Conway-Coxeter frieze patterns.
\end{Theorem}

In the cases of frieze patterns over $\mathcal{O}_d$ for $d\in \{-1,-2,-3,-7,-11\}$ it seems to be a hard problem to classify all non-zero frieze patterns and each of these cases might need different ideas.
After posting our preprint to the arxiv, it has been stated in \cite{FKST23} that Theorem \ref{thmint_3} could also be obtained by a geometric argument.

\section*{Acknowledgements} While completing this work, the third author has been supported by a Riemann fellowship at Hannover University, to whom goes his gratitutude for the financial support and the great hospitality.

\section{Subrings of $\overline{\mathbb{Q}}$ with finitely many units}\label{sec:subrings}

In this section we give a proof of:

\begin{Theorem}\label{thmunits}
Let $R\le \overline{\mathbb{Q}}$ be a subring with finitely many units. Then $R=\mathbb{Z}$ or there exists an integer $d\in\mathbb{Z}_{<0}$ such that $R$ is an order in $\mathbb{Q}(\sqrt{d})$.
\end{Theorem}

Let us, as a first approximation, show that a subring of $\overline{\mathbb{Q}}$ with finitely many units is contained in an imaginary quadratic field.

\begin{Proposition}\label{Sdeg}
Let $R\le\overline{\mathbb{Q}}$ be a subring and denote by $\quot(R)$ the field of fractions of $R$.
If $R$ has finitely many units, then $\quot(R)$ is either $\mathbb{Q}$ or an imaginary quadratic field.
\end{Proposition}

\begin{proof}
We can assume that the field extension $\quot(R)/\mathbb{Q}$ has finite degree; in fact if $\quot(R)/\mathbb{Q}$ has infinite degree, choose more than two $\mathbb{Q}$-linearly independent elements $\alpha_1,\ldots,\alpha_n$ of $R$ and replace $R$ by $\mathbb{Z}[\alpha_1,\ldots,\alpha_n]\le R$ in the argument.

By the primitive element theorem, we have an element $\vartheta_0$ in $\quot(R)$ such that $\mathbb{Q}(\vartheta_0)=\quot(R)$. Observe that $\frac{\quot(R)}{R}$ is a torsion group. Hence by multiplying by a sufficiently large integer $N$, we find that $\vartheta:=N\cdot \vartheta_0$ is both in $R$ and $\mathbb{Z}[\vartheta]$ is an order in $\quot(R)$. 

Dirichlet's unit theorem (see \cite[Thm.\ 12.12]{Neukirch} for a version including orders) states that the rank of the group of units of $\mathbb{Z}[\vartheta]$ is $r+s-1$, where $r$ is the number of real embeddings and $s$ the number of conjugate pairs of complex embeddings of $\mathbb{Q}(\vartheta)=\quot(R)$.

If $r+2s=[\mathbb{Q}(\vartheta):\mathbb{Q}]>2$, then $r+s-1>0$, thus the group of units $R^\times$ is infinite because it contains $\mathbb{Z}[\vartheta]^\times$.
On the other hand, if $\mathbb{Q}(\vartheta)$ is real quadratic, then $r=2$, $s=0$ and $r+s-1>0$ as well. This gives the desired conclusion.
\end{proof}

Before concluding the proof of Theorem \ref{thmunits},
let us prove two simple facts. All rings here are commutative with unit. Given two subrings $R_1,R_2$ of a ring $R_3$, we denote by $R_1 \cdot R_2$ the ring generated by the two rings in $R_3$. It is nothing else than the natural image of the map
$$R_1 \otimes_{\mathbb{Z}} R_2 \to R_3,
$$
given by the bilinear form obtained by the multiplication on $R_3$. 
\begin{Proposition} \label{finite index}
$(a)$ Let $R_1,R_2$ be two subrings of $R_3$. Suppose that $R_1 \cap R_2$ has finite index in $R_2$. Suppose furthermore that for every positive integer $n$, we have that $R_1/nR_1$ is finite. Then $R_1$ has finite index in $R_1 \cdot R_2$. \\
$(b)$ Suppose that $R_1 \subseteq R_2$ is an inclusion of rings with finite index (as abelian groups). Then $R_1^{\times} \subseteq R_2^{\times}$ has also finite index (as abelian groups).
\begin{proof}
For part $(a)$. The natural map
$$R_1 \otimes_{\mathbb{Z}} R_2 \twoheadrightarrow \frac{R_1 \cdot R_2}{R_1}
$$
factors through
$$R_1 \otimes_{\mathbb{Z}}\frac{R_2}{R_1 \cap R_2} \twoheadrightarrow \frac{R_1 \cdot R_2}{R_1}.
$$
Now the second factor is a finite abelian group by assumption. Let $n$ be its order. Then the map further factors through
$$\frac{R_1}{nR_1} \otimes_{\mathbb{Z}}\frac{R_2}{R_1 \cap R_2} \twoheadrightarrow \frac{R_1 \cdot R_2}{R_1}.
$$
Now the source is a finite abelian group and so also the target must be finite, as desired. 

For part $(b)$. Observe that the abelian group $R_2/R_1$ is naturally an $R_1$-module. Denote by $J$ the kernel of the natural map $R_1 \to \End_{\text{ab.gr.}}(R_2/R_1)$. It is clearly an ideal of $R_1$. But as a matter of fact it is already an ideal of $R_2$. Indeed $rJR_2=JrR_2 \subseteq R_1$, for each $r$ in $R_2$. Furthermore $J$ has finite index in $R_1$ (the endomorphism ring of a finite abelian group is certainly finite), which has finite index in $R_2$. It follows that $J$ is an ideal of $R_2$ with finite index. Hence we have that the subgroup
$$\text{ker}(R_2^{\times} \to (R_2/J)^{\times})
$$
has also finite index in $R_2^{\times}$ (the target group is the unit group of a finite ring). But by construction this consists of elements of the form $1+j$ with $j$ in $J$ and hence of $R_1^{\times}$ (the inverse of $1+j$ is a priori in $R_2$, but it has to be still in $1+J$ as one can see reducing $R_2$ modulo $J$, therefore the inverse is also in $R_1$), which therefore in particular must have finite index (since a subgroup does), as desired. 
\end{proof}
\end{Proposition}

\begin{Proposition} \label{more finite}
Let $R$ be a subring of a number field $L$. Let $n$ be a positive integer. Then $R/nR$ is finite.
\end{Proposition}

\begin{proof}
For an abelian group $A$, the property that $A/nA$ is finite for all $n$ is equivalent to the same property over primes $n:=p$. Now suppose that $A$ is a torsion free abelian group, such that all its finitely generated subgroups have rank at most $d$, for a uniform positive integer $d$. Then we claim that $A/pA$ must be finite for all prime number $p$, and more precisely it must have dimension at most $d$, for each $p$. Suppose not. Then we can find $d+1$ elements $a_1, \ldots,a_{d+1}$ in $A$, which are linearly independent vectors in $A/pA$. Let now $H:=\langle a_1, \ldots, a_{d+1} \rangle$ be the group generated by these elements. It must be that $H/pH$ is a $d+1$-dimensional vector space, since its image in $A/pA$ is already $d+1$-dimensional. It follows that the finitely generated group $H$ surjects onto $\mathbb{F}_p^{d+1}$, hence its rank has to be at least $d+1$, a contradiction. This gives the desired conclusion for any such $A$.

We now reach the desired conclusion for $R$, since $R$ is a subgroup of a finite dimensional $\mathbb{Q}$-vector space, $L$. And if $d$ denotes the dimension of this space then clearly every finitely generated subgroup of $L$ must have rank at most $d$. This ends the proof. 
\end{proof}

To conclude the proof of Theorem \ref{thmunits} we have the following. 
\begin{Proposition}\label{prop:unit}
Let $R$ be a subring of a number field $L$. Suppose that $R^{\times}$ is finite. Then $R$ is contained in $\mathcal{O}_K$ for some imaginary quadratic field $K$ inside $L$.
\end{Proposition}

\begin{proof}
Thanks to Proposition \ref{Sdeg}, we already know that there must be $K \subseteq L$ a sub-extension that is imaginary quadratic and contains $R$. Hence we are reduced to proving the statement in the case that $K=L$ is an imaginary quadratic field. 

Let us first assume that $R$ is not a subring of $\mathbb{Q}$. This means that $R \cap \mathcal{O}_K$ is a finite index subring of $\mathcal{O}_K$. It follows that $R':=R \cdot \mathcal{O}_K$ contains $R$ as a finite index subring thanks to Proposition \ref{finite index} combined with Proposition \ref{more finite}. It follows that $\#R'^{\times}<\infty$, since $R^{\times}$ has finite index therein, by Proposition \ref{finite index}, and it is finite by assumption. Hence we are reduced to prove the desired conclusion for $R'$, which, this time, contains $\mathcal{O}_K$. Suppose by contradiction that there exists $\alpha$ in $R'$ which is not contained in $\mathcal{O}_K$. Let us factor the principal $\mathcal{O}_K$-module $(\alpha)$ in prime factors
$$(\alpha)=\prod_{\mathfrak{p} \in S^{+}}\mathfrak{p}^{n(\mathfrak{p})} \cdot \prod_{\mathfrak{q} \in S^{-}}\mathfrak{q}^{n(\mathfrak{q})},
$$
where for all $\mathfrak{p}$ in $S^{+}$ we have that $n(\mathfrak{p})>0$ and for all $\mathfrak{q}$ in $S^{-}$ we have that $n(\mathfrak{q})<0$. The assumption that $\alpha$ is not in $\mathcal{O}_K$ is equivalent to saying that $S^{-}$ is not empty. 

Now, since the class group of $\mathcal{O}_K$ is finite (\cite[Theorem 6.3]{Neukirch}), we can raise $\alpha^h$ for some positive integer $h$, in such a way that each ideal in the expression becomes principal. We conclude that we can write
$$\alpha^h:=\frac{\gamma_1}{\gamma_2},
$$
with $\gamma_1,\gamma_2$ coprime elements of $\mathcal{O}_K$ (and hence also of $R'$) such that $\gamma_2$ has positive valuation at some prime. Since $\gamma_1,\gamma_2$ are coprime, we can find $x,y$ in $\mathcal{O}_K$ such that
$$1=x\gamma_1+y\gamma_2.
$$
It follows that $x\alpha^h+y=\frac{1}{\gamma_2}$, where the left hand side is visibly in $R'$. Hence we have that $\frac{1}{\gamma_2}$ is an element of $R'$, and therefore 
$$\gamma_2 \in {R'}^{\times},
$$
is an element of infinite order, since its image under the valuation map at one place has infinite order. This gives the desired conclusion. 

We are left with the case where $R$ is contained in $\mathbb{Q}$, where we replicate the same argument as above, this time with $h=1$. This ends the proof.
\end{proof}

\begin{Example}
Let $d=-13$, $\tau:=\sqrt{d}$, $\alpha=(-2+5\tau)/{47}$ and consider the ring $R=\mathbb{Z}[\alpha]$.
We compute a unit in $R$ of infinite order following the proof of Proposition \ref{prop:unit}.
Denote $\mathcal{O}$ the maximal order in $\quot(R)$. Then in $\mathcal{O}$ we have the following factorizations of ideals into prime ideals:
\[
(-2+5 \tau) = (7,\: 1+\tau)\cdot (47,\: 9+\tau), \quad (47) = (47,\: 38+\tau)\cdot (47,\: 9+\tau).
\]
Thus we have
\[ (\alpha) = (7,\: 1+\tau) \cdot (47,\: 38+\tau)^{-1}. \]
With $h=2$ we get principal ideals
\[ (7,\: 1+\tau)^2 = (6-\tau), \quad (47,\: 38+\tau)^2 = (-34+9\tau), \quad \alpha^h = \frac{6-\tau}{-34+9\tau} = \frac{\gamma_1}{\gamma_2}, \]
and indeed, $\gamma_1=6-\tau$ and $\gamma_2=-34+9\tau$ are coprime in $\mathcal{O}$.
The element $1/\gamma_2$ is a unit in $R\cdot \mathcal{O}$ but not in $R$. Since $R$ has finite index in $R\cdot \mathcal{O}$, it suffices to compute powers of $1/\gamma_2$:
the third power of $1/\gamma_2$ is equal to $(68102-21735\tau)/47^6$ which turns out to be a unit in $R$:
$$ (68102+21735\tau) \cdot (68102-21735\tau)/47^6 = 1 $$
is an equality in $\mathbb{Z}[\alpha]$ since
\begin{eqnarray*}
- 68102-21735\tau &=& 5900521\alpha^6+12400457\alpha^5+2969, \\
( - 68102+21735\tau)/47^6 &=& -1215601\alpha^6-339167\alpha^5-4689.
\end{eqnarray*}
The norms of these elements are not $1$; they have infinite order in the group $R^\times$.
\end{Example}

\section{Frieze patterns and algebraic numbers}\label{sec:algnum}

In this section we collect some of the fundamental definitions in the theory of frieze patterns.
We then study in detail frieze patterns over algebraic number fields. As a main result we obtain in this section a proof of Theorem \ref{thm:Rcirc} from the introduction.

\begin{Definition} \label{def:friezepattern}
A \emph{non-zero frieze pattern} $\mathcal{C}$ of height $n$ over a subset $R\subseteq K$ of a field $K$ is an infinite array of the form
\[
\begin{array}{ccccccccccc}
 & & \ddots & & & &\ddots  & & & \\
0 & 1 & c_{i-1,i+1} & c_{i-1,i+2} & \cdots & \cdots & c_{i-1,n+i} & 1 & 0 & & \\
& 0 & 1 & c_{i,i+2} & c_{i,i+3} & \cdots & \cdots & c_{i,n+i+1} & 1 & 0 & \\
& & 0 & 1 & c_{i+1,i+3} & c_{i+1,i+4} & \cdots & \cdots & c_{i+1,n+i+2} & 1 & 0 \\
 & & & & \ddots  & & & &\ddots  & 
\end{array}
\]
with $c_{i,j}\in R\setminus \{0\}$ such that each adjacent $2\times 2$-submatrix has determinant 1, that is,
$c_{i,j}c_{i+1,j+1}-c_{i,j+1}c_{i+1,j}=1$ for all $i\in \mathbb{Z}$ and $i+1\le j\le n+i+2$ (with $c_{i,i+1}=1=c_{i,n+i+2}$).
We sometimes write $\mathcal{C}=(c_{i,j})$ for such a frieze pattern.
\end{Definition}

Every non-zero frieze pattern $\mathcal{C}$ as above is {\em tame} (that is, all neighbouring $3\times 3$-determinants are 0); a proof of this fact can be found in \cite[Proposition 2.6]{CHJ20}.
\smallskip

Note that for a non-zero frieze pattern all entries are uniquely determined by the entries in the first diagonal $(\ldots,c_{i-1,i+1},c_{i,i+2},c_{i+1,i+3},\ldots)$. Following Conway and Coxeter \cite{CC73} we call this the {\em quiddity sequence} of the frieze pattern.
It is known that non-zero frieze patterns satisfy a glide reflection, in particular the quiddity sequence is invariant under translation by $n+3$ steps. So the entire frieze pattern is determined by the {\em quiddity cycle} $(c_{0,2},c_{1,3},\ldots,c_{n+2,n+4})$ (or any shift thereof). 
\smallskip

We address now the fundamental question of whether for a given ring $R$ there are finitely many or infinitely many non-zero frieze patterns over $R$ in each height.
The following known result shows that such a ring should have finitely many units.

\begin{Proposition}\label{prop:units}
Let $R$ be a ring containing infinitely many pairs $(a,b)\in R\setminus\{0\}$ with $ab=2$.
Then there are infinitely many non-zero frieze patterns of height $1$ over $R$.
In particular this holds for a ring $R$ with infinitely many units and characteristic $\ne 2$. 
\end{Proposition}
\begin{proof}
Each pair $(a,b)$ with $ab=2$ gives the quiddity cycle $(a,b,a,b)$ of length $4$, and this yields a non-zero frieze pattern over $R$ of height 1.
\end{proof}

From now on we consider subrings $R$ of the field of complex numbers $\mathbb{C}$. In the light of Proposition \ref{prop:units} it is relevant for our purposes to determine whether or not such a subring contains finitely or infinitely many units. This seems to be a difficult question in general if a subring $R$ contains transcendental numbers.
So we restrict in this paper to subrings of the field $\overline{\mathbb{Q}}$ of algebraic numbers.

We shall need the following notion frequently.

\begin{Definition} \label{def:friezesubring}
Let $R\le \mathbb{C}$ be a subring. We define $R^{\circ}$ to be the subring of $R$ generated by all entries of all non-zero frieze patterns over $R$ and call it the {\em frieze subring of $R$}.
\end{Definition}

We can then state a strong and maybe surprising consequence of Theorem \ref{thmunits}.

\begin{Corollary}\label{cor:Qd}
Let $R\le \mathbb{C}$ be a subring such that the number of non-zero frieze patterns over $R$ is finite in each height. 
If the frieze subring $\cR$ does not contain a transcendental number, then $\cR=\mathbb{Z}$ or there exists an integer $d\in \mathbb{Z}$, $d<0$ such that $\cR$ is an order in $\mathbb{Q}(\sqrt{d})$.
\end{Corollary}
\begin{proof}
Since the number of non-zero frieze patterns in each height is finite, it is in particular finite in height $1$; this implies that $\cR$ has finitely many units by Proposition \ref{prop:units}.
If $\cR$ does not contain a transcendental number, then $\cR\le \overline{\mathbb{Q}}$, thus by Theorem \ref{thmunits},
$\cR=\mathbb{Z}$ or 
$\cR$ is an order in an imaginary quadratic number field.
\end{proof}

This corollary is the key step to proving our first main result Theorem \ref{thm:Rcirc} from the introduction. We restate this theorem here for the convenience of the reader.

\begin{Theorem} \label{thm:Rcircv2}
Let $R\le \mathbb{C}$ be a subring. 
If the frieze subring $\cR$ is contained in  $\overline{\mathbb{Q}}$, then there are finitely many non-zero frieze patterns over $R$ in each positive height if and only if
\begin{enumerate}
 \item[{(i)}] $\cR=\mathbb{Z}$,\quad or
 \item[{(ii)}] $\cR$ is an order in $\mathbb{Q}(\sqrt{d})$ with $d\in \{-1,-2,-3,-7,-11\}$.
\end{enumerate}
\end{Theorem}

\begin{proof}
Let $R^{\circ}\le \overline{\mathbb{Q}}$ and suppose that there are only finitely many non-zero frieze patterns over $R$ in each height. By Corollary \ref{cor:Qd} we know that $\cR=\mathbb{Z}$ or
$\cR$ is an order in $\mathbb{Q}(\sqrt{d})$ for some rational integer $d<0$.
If $R^{\circ}=\mathbb{Z}$ or $R^{\circ}\subseteq \mathcal{O}_d$ with $d\in \{-1,-2,-3,-7,-11\}$, then we are in case $(i)$ or case $(ii)$ of the theorem and there is nothing to show.
For every negative rational integer $d\not\in \{-1,-2,-3,-7,-11\}$ we show in Theorem \ref{thm:imaginary} that $\mathcal{O}_d^{\circ}=\mathbb{Z}$, and we are in case $(i)$ of the
theorem.

We now prove the converse. If $R^{\circ}=\mathbb{Z}$ then there are finitely many non-zero frieze patterns in each height by the work of Conway-Coxeter \cite{CC73} and Fontaine \cite{F14}; we give an independent proof of this fact in Section \ref{sec:reduction} below. 
As mentioned in the introduction, it is known that for every rational integer $d<0$ the ring of integers $\mathcal{O}_d$ is a discrete subset of $\mathbb{C}$ and hence there are finitely many non-zero frieze patterns over $\mathcal{O}_d$ in each height by \cite[Corollary 3.8]{CH17}.
\end{proof}

We close this section with general observations on frieze patterns over rings. The first one abstracts one of the steps in the proof of Proposition \ref{prop:unit}.

\begin{Lemma}\label{ab_units}
Let $S$ be a Dedekind domain with torsion class group and $S\le R \le \quot(S)$ a ring. If $S\ne R$, then $S^\times$ has infinite index in $R^\times$.
\end{Lemma}

\begin{proof}
One can apply the proof of Proposition \ref{prop:unit} almost verbatim. The only additional observation is that in the step of that proof where we invoked the finiteness of the class group we only needed it to be a torsion group. Following that argument (with same notation) we find an element $\gamma_2$ which has positive valuation at some prime of $S$ and is invertible in $R$. This means that the map
$$\text{div}_R:\quot(S)^\times \to \bigoplus_{\mathfrak{p} \ \text{maximal ideal in} \ S} \mathbb{Z}\cdot \mathfrak{p},
$$
sending an element of $\quot(S)^\times$ in its valuation vector, sends $S^\times$ to $0$ and $\gamma_2$ to a non-zero element, hence necessarily of infinite order. It follows that $\frac{R^{\times}}{S^\times}$ has a quotient of infinite order and therefore it is infinite, as desired.
\end{proof}

\begin{Remark}
If $S$ is finite in Lemma \ref{ab_units}, then it is well known that $S$ is a field. Hence in this case $S=\quot(S)$ and there is no ring $R\ne S$ between $S$ and $\quot(S)$.
\end{Remark}

\begin{Corollary} \label{cor:pidunit}
Let $S$ is a principal ideal domain and $S\le R \le \quot(S)$ a ring. If $S\ne R$, then $R$ has infinitely many units.
\end{Corollary}
\begin{proof}
This follows at once from Lemma \ref{ab_units} and the fact that principal ideal domains are precisely the Dedekind domains with trivial (and so in particular torsion) class group.
\end{proof}

\begin{Proposition} \label{prop:pid}
Let $S$ be a principal ideal domain and $\quot(S)$ be of characteristic $\ne 2$.
If $S\le R\le \quot(S)$ is a ring such that the number of non-zero frieze patterns over $R$ is finite in each height, then $R=S$.
\end{Proposition}
\begin{proof}
Assume for a contradiction that $\frac{a}{b}\in R\setminus S$ where $a,b\in S$.
By Corollary \ref{cor:pidunit}, $R$ has infinitely many units.
But this contradicts Proposition \ref{prop:units}: $R$ can only have finitely many units because there are only finitely many non-zero frieze patterns over $R$ of height 1 by assumption.
\end{proof}

As the special case $S=\mathbb{Z}$ in Proposition \ref{prop:pid} we obtain that $\mathbb{Z}$ is the only subring of the rational numbers having finitely many non-zero frieze patterns in each height.

\begin{Corollary}
Let $R\le \mathbb{Q}$ be a subring such that the number of non-zero frieze patterns over $R$ is finite in each height.
Then $R=\mathbb{Z}$.
\end{Corollary}

The following result is well-known, we include a proof for the convenience of the reader.

\begin{Proposition}\label{discreteQd}
\begin{enumerate}
    \item Let $R\le \mathbb{C}$ be a discrete subring. Then there exists a rational integer $d<0$ such that $R\le \mathbb{Q}(\sqrt{d})$.
    \item Let $z\in \overline{\mathbb{Q}}\le \mathbb{C}$ be such that $\mathbb{Z}[z]$ is discrete with respect to the topology induced by $\mathbb{C}$. Then $\mathbb{Z}[z]$ is an order in $\mathbb{Q}(z)$.
\end{enumerate}
\end{Proposition}
\begin{proof}
(1) 
Since $R$ is discrete, as an abelian subgroup of $\mathbb{C}$ its rank is at most two. If the rank is one, then $R = \mathbb{Z}$.
Otherwise there exists an $\omega\in R$ such that $R=\langle 1,\omega\rangle$ as an abelian group. Thus there are $a,b\in\mathbb{Z}$ such that $\omega^2 = a\omega+b$. Hence $\omega$ is an algebraic integer and the field 
$\mathbb{Q}(\omega)$ has degree 2 over $\mathbb{Q}$. Moreover, $\mathbb{Q}(\omega) \not\subset \mathbb{R}$ because $R$ has rank two. So $\mathbb{Q}(\omega)=\mathbb{Q}(\sqrt{d})$ for some $d<0$ and it follows that $R\le \mathbb{Q}(\omega)\le
\mathbb{Q}(\sqrt{d})$, as claimed.

(2)
We proceed as in (1) for $R=\mathbb{Z}[z]$ and obtain an algebraic integer $\omega$ such that $\mathbb{Z}[z]=\langle 1,\omega\rangle$. Thus $z$ is integral as well and hence $\mathbb{Z}[z]$ is an order in $\mathbb{Q}(z)$.
\end{proof}

\section{Non-zero frieze patterns over integers}
\label{sec:reduction}

The goal of this section is to give a self-contained proof of the known classification of frieze patterns with entries from the set of non-zero rational integers. It turns out that in addition to the classic Conway-Coxeter frieze patterns over the positive rational integers only very few new frieze patterns occur. We shall give an independent proof
of this result below. These non-zero frieze patterns over $\mathbb{Z}$ will appear again in the next section when we study frieze patterns over rings of quadratic integers. 
\smallskip

We first provide some observations which hold more generally for non-zero frieze patterns over the set of real numbers with absolute value at least $\frac{1}{\sqrt{2}}\approx 0,7071$. 

\begin{Lemma} \label{lemma:signs}
Let $\mathcal{C}=(c_{i,j})$ be a non-zero frieze pattern as in Definition \ref{def:friezepattern} and suppose that
$c_{i,j}\in \mathbb{R}_{> a}\cup \mathbb{R}_{< -a}$ where $a=\frac{1}{\sqrt{2}}$. We set 
$$\varepsilon_{i,j}=\left\{ 
\begin{array}{rl} 1 & \mbox{if $c_{i,j}>0$} \\ -1 & \mbox{if $c_{i,j}<0$}
\end{array} \right.. 
$$
Then the following hold.
\begin{enumerate}
\item[{(a)}] For all $i\in\mathbb{Z}$ and $i+2\le j\le n+i+1$ 
we have 
$$\varepsilon_{i,j}\varepsilon_{i+1,j}\varepsilon_{i,j+1}\varepsilon_{i+1,j+1}=1.
$$
\item[{(b)}] For all $i\in\mathbb{Z}$ and $1\le \ell \le n+1$ 
we have 
$$\varepsilon_{i,i+2}\varepsilon_{i+1,i+2}\varepsilon_{i,i+\ell+1}\varepsilon_{i+1,i+\ell+1}=1.
$$
\item[{(c)}] We have $|\{\varepsilon_{i,i+2}\,|\,i\in \mathbb{Z}\}|=1$.
\end{enumerate}
\end{Lemma}

\begin{proof}
$(a)$
Note that the assertion in the lemma claims that in each neighbouring $2\times 2$-submatrix 
$\begin{array}{cc} c_{i,j} & c_{i,j+1} \\ c_{i+1,j} & c_{i+1,j+1} \end{array}$ of the frieze pattern there is an even number of positive entries (and hence also an even number of negative entries). 
So it suffices to show that all configurations with an odd number of positive entries can not occur.
By elementary arithmetic each of the configurations of signs
$$\begin{array}{cc} + & + \\ + & - \end{array}
~~~\mbox{\hskip0.2cm and\hskip0.2cm}~~~
\begin{array}{cc} - & + \\ + & + \end{array}
~~~\mbox{\hskip0.2cm and\hskip0.2cm}~~~ 
\begin{array}{cc} + & - \\ - & - \end{array}
~~~\mbox{\hskip0.2cm and\hskip0.2cm}~~~ 
\begin{array}{cc} - & - \\ - & + \end{array}
$$ 
would yield a negative number as determinant of a $2\times 2$-submatrix, contradicting the condition $c_{i,j}c_{i+1,j+1}-c_{i,j+1}c_{i+1,j}=1$ in Definition \ref{def:friezepattern}.
The remaining cases 
$$\begin{array}{cc} + & - \\ + & + \end{array}
~~~\mbox{\hskip0.2cm and\hskip0.2cm}~~~
\begin{array}{cc} + & + \\ - & + \end{array}
~~~\mbox{\hskip0.2cm and\hskip0.2cm}~~~ 
\begin{array}{cc} - & + \\ - & - \end{array}
~~~\mbox{\hskip0.2cm and\hskip0.2cm}~~~ 
\begin{array}{cc} - & - \\ + & - \end{array}
$$ 
give a positive determinant, but since $|c_{i,j}|> \frac{1}{\sqrt{2}}$ by assumption this determinant is strictly bigger than 1, again contradicting the condition in Definition \ref{def:friezepattern}.
\smallskip

\noindent
$(b)$ We proceed by induction on $\ell$. For $\ell=1$ we have
$$\varepsilon_{i,i+2}\varepsilon_{i+1,i+2}\varepsilon_{i,i+2}\varepsilon_{i+1,i+2}
=(\varepsilon_{i,i+2}\varepsilon_{i+1,i+2})^2 =1.
$$ 
Let $\ell>1$. By part $(a)$ we have
$$\varepsilon_{i,i+\ell}\varepsilon_{i+1,i+\ell}\varepsilon_{i,i+\ell+1}\varepsilon_{i+1,i+\ell+1}=1.
$$
Moreover, by induction hypothesis we can assume that
$$\varepsilon_{i,i+2}\varepsilon_{i+1,i+2}\varepsilon_{i,i+\ell}\varepsilon_{i+1,i+\ell}=1.
$$
Putting these equation together yields
$$\varepsilon_{i,i+2}\varepsilon_{i+1,i+2}\varepsilon_{i,i+\ell+1}\varepsilon_{i+1,i+\ell+1}=1.
$$

\noindent
$(c)$ Note that $\varepsilon_{i+1,i+2}=1$ since $c_{i+1,i+2}=1$ and similarly
$\varepsilon_{i,n+i+2}=1$ since $c_{i,n+i+2}=1$ (see Definition \ref{def:friezepattern}).  
Then part $(b)$ for $\ell=n+1$
yields
\begin{equation} \label{eq:epsilon}
\varepsilon_{i,i+2} = \varepsilon_{i+1,n+i+2}\mbox{\hskip0.3cm for all $i\in \mathbb{Z}$}.
\end{equation}
The frieze pattern $\mathcal{C}$ is non-zero and hence tame (see the remark after Definition \ref{def:friezepattern}). Every tame frieze pattern satisfies a glide symmetry, more precisely we have $c_{i,j}=c_{j,n+i+3}$ for all $i,j$ (a proof of this known fact can be found in \cite[Theorem 2.12]{CHJ20}). This implies that for all $i\in \mathbb{Z}$ we have
$c_{i-1,i+1} = c_{i+1,n+i+2}$.
Together with equation (\ref{eq:epsilon}) we deduce that
$$\varepsilon_{i-1,i+1} = \varepsilon_{i+1,n+i+2} = \varepsilon_{i,i+2}\mbox{\hskip0.2cm for all 
$i\in \mathbb{Z}$}
$$
and this proves the claim in part $(c)$.
\end{proof}

We now state the classification of non-zero frieze patterns over $\mathbb{Z}$, combining the classic result of Conway and Coxeter \cite{CC73} and the extension by Fontaine \cite{F14}.
With our above observations we can provide an independent and self-contained proof.
\smallskip

We recall a general construction on frieze patterns with odd height: multiplying the diagonal containing the quiddity sequence and then every second diagonal with $-1$ again gives a frieze pattern (this follows immediately from Definition \ref{def:friezepattern} since the determinant of every neighbouring $2\times 2$-submatrix is unchanged). 

\begin{Theorem}\label{thm:Fontaine}
Let $\mathcal{C}=(c_{i,j})$ be a non-zero frieze pattern over the rational integers. 
Then one of the following assertions hold.
\begin{enumerate}
\item[{(i)}] $\mathcal{C}$ is a Conway-Coxeter frieze pattern, that is, all entries are positive integers. 
\item[{(ii)}] The height of $\mathcal{C}$ is odd and $\mathcal{C}$ is obtained from a Conway-Coxeter frieze pattern by multiplying every second diagonal by $-1$.
\end{enumerate}
\end{Theorem}

\begin{proof}
From Lemma \ref{lemma:signs}\,$(c)$ we know that the entries $c_{i,i+2}$ in the quiddity sequence all have the same sign. If all  these quiddity entries are positive integers then it follows from \cite[Equation (6.6)]{Cox71} that all entries in the frieze pattern are positive integers. So $\mathcal{C}$ is a Conway-Coxeter frieze pattern. 

Now suppose that all quiddity entries $c_{i,i+2}$ are negative integers. By glide symmetry it follows that $c_{i,n+i+1} = c_{n+i+1,n+i+3}$ are negative integers for all $i\in \mathbb{Z}$ (where $n$ denotes the height of $\mathcal{C}$). On the other hand, $\varepsilon_{i,i+2}=-1$ by assumption, hence we deduce from Lemma \ref{lemma:signs}\,$(b)$ that
$$\varepsilon_{i,i+\ell+1} = - \varepsilon_{i+1,i+\ell+1}\mbox{\hskip0.2cm for all $i\in \mathbb{Z}$ and $1\le \ell\le n+1$}. $$
This implies that the signs on each diagonal in the frieze pattern are constant, namely all entries $c_{i,j}$ are positive if $j-i$ is odd and negative if $j-i$ is even. Since we have observed above that the numbers $c_{i,n+i+1}$ are negative we conclude that $n+i+1-i=n+1$ is even, thus the height $n$ of $\mathcal{C}$ is odd. 
We can therefore apply the standard construction on frieze patterns of odd height, multiplying every second diagonal by $-1$. Then we obtain a frieze pattern with all entries positive integers, that is, a Conway-Coxeter frieze pattern. This means that assertion $(ii)$ of the theorem holds for $\mathcal{C}$. 
\end{proof}

\section{Non-zero frieze patterns over rings of quadratic integers} 
\label{sec:finiteness}

In this section we study non-zero frieze patterns over rings of integers $\mathcal{O}_d$ in quadratic number fields $\mathbb{Q}(\sqrt{d})$, where $d\in \mathbb{Z}\setminus \{0,1\}$ is a square-free integer. 

It seems to be a subtle problem to describe all non-zero frieze patterns over the Gaussian integers $\mathcal{O}_{-1}=\mathbb{Z}[i]$ and the Eisenstein integers $\mathcal{O}_{-3}=\mathbb{Z}[\frac{1+\sqrt{-3}}{2}]$. 
As the main result of this section we will show that among the imaginary quadratic integers there are only finitely many such difficult cases. Apart from these few cases, all non-zero frieze patterns over imaginary quadratic integers are known. For the proof we shall need a different viewpoint on frieze patterns and some useful results on reductions of quiddity cycles. All this is based on certain $2\times 2$-matrices which play a fundamental role in the theory of frieze patterns. We briefly recall the necessary background here.
\smallskip
 
Let $K$ be a field. For any $c\in K$ we define the $2\times 2$-matrix
$$\eta(c)=\begin{pmatrix} c & -1  \\ 1 & 0 \end{pmatrix}. $$
Then a sequence $(c_{0,2},c_{1,3},\ldots,c_{n+2,n+4})$ of numbers is the quiddity cycle of a tame frieze pattern (of height $n$) if and only if 
$$\prod_{i=0}^{n+2} \eta(c_{i,i+2}) = \begin{pmatrix} -1 & 0 \\ 0 & -1 \end{pmatrix} $$
(for a proof of this known fact see for instance \cite[Proposition 2.4]{CH17}). In general, a sequence $(c_1,\ldots,c_m)\in \mathbb{C}^m$ is called a {\em quiddity cycle} if
$\prod_{i=1}^m \eta(c_i) = -\mathrm{id}$ is the negative of the identity matrix.
\smallskip

For later use we collect some results from \cite{CH17} stating that every quiddity cycle must contain some small entries and giving reduction formulae for quiddity cycles. 

\begin{Lemma}  \label{lem:small}
\begin{enumerate}
\item[{(a)}] 
Let $(c_1,\ldots,c_m)\in \mathbb{C}^m$ such that $\prod_{j=1}^m \eta(c_j)$ is a scalar multiple of the identity matrix.
Then there are two different indices $j,k\in\{1,\ldots,m\}$ with $|c_j|<2$ and $|c_k|<2$.
\item[{(b)}] 
For all $a,b\in \mathbb{C}$ we have
\begin{enumerate}
\item[{(i)}] 
$\eta(a)\eta(1)\eta(b) = \eta(a-1)\eta(b-1)$
\item[{(ii)}] $\eta(a)\eta(-1)\eta(b) = -\eta(a+1)\eta(b+1)$.
\item[{(iii)}] $\eta(a)\eta(0)\eta(b) = -\eta(a+b)$.
\end{enumerate}
\item[{(c)}] 
Let $(c_1,\ldots,c_m)\in \mathbb{C}^m$ be a quiddity cycle with $m>3$.
Then there are two indices $j,k\in\{1,\ldots,m\}$ with $|j-k|>1$ and $\{j,k\}\ne\{1,m\}$ such that $|c_j|<2$ and $|c_k|<2$.
\end{enumerate}
\end{Lemma}

\begin{proof}
The statements in parts $(a)$ and $(b)$ are \cite[Corollary 3.3]{CH17} and \cite[Proposition 4.1]{CH17}, respectively. The assertion in part $(c)$ is stated in \cite[Corollary 6.3]{CH17} for quiddity cycles over $\mathbb{Z}$. However, the integrality assumption is only used in the proof for dealing with the case $m=4$. The argument for $m=4$ can easily be generalized to arbitrary quiddity cycles over $\mathbb{C}$. Namely, every such quiddity cycle is of the form $(c_1,\frac{2}{c_1},c_1,\frac{2}{c_1})$; if $|c_1|<2$, we are done, and if $|c_1|\ge 2$ then $|\frac{2}{c_1}|\le 1<2$ and we are also done. 
\end{proof}

We now state the main result of this section.

\begin{Theorem} \label{thm:imaginary}
Let $d$ be a negative square-free rational integer. Then
\[ \mathcal{O}_d^{\circ} = \begin{cases}
 \mathcal{O}_d & \text{if } d\in \{-1,-2,-3,-7,-11\} \\
 \mathbb{Z} & \text{else}.
\end{cases} \]
\end{Theorem}

\begin{Remark}
Note that the above theorem states that for any negative square-free integer $d\not\in \{-1,-2,-3,-7,-11\}$ the only non-zero frieze patterns over $\mathcal{O}_d$ are the Conway-Coxeter frieze patterns and the related twisted Conway-Coxeter frieze patterns (appearing in Theorem \ref{thm:Fontaine}). For the values $d\in \{-1,-2,-3,-7,-11\}$ there exist non-integral frieze patterns as shown in the subsequent proof.
\end{Remark}

\begin{figure}[h]
\centering
\begin{tabular}{c|l}
$d$ & quiddity cycle \\
\hline
$-1$ & $(1+\omega_{-1},1-\omega_{-1},1+\omega_{-1},1-\omega_{-1})$ \\
$-2$ & $(\omega_{-2},\overline{\omega_{-2}},\omega_{-2},\overline{\omega_{-2}})$ \\
$-3$ & $(\omega_{-3},2\overline{\omega_{-3}},\omega_{-3},2\overline{\omega_{-3}})$ \\
$-7$ & $(\omega_{-7},\overline{\omega_{-7}},\omega_{-7},\overline{\omega_{-7}})$ \\
$-11$ & $(\omega_{-11},\overline{\omega_{-11}},\omega_{-11},\overline{\omega_{-11}},\omega_{-11},\overline{\omega_{-11}})$ \\
\end{tabular}
\caption{Certain quiddity cycles in $\mathcal{O}_d$}
\label{fig:quid}
\end{figure}

\begin{proof} The proof consists of two parts.
First, for $d\in \{-1,-2,-3,-7,-11\}$, we present explicit non-integral frieze patterns over
$\mathcal{O}_d$, showing that $\mathcal{O}_d^{\circ}=\mathcal{O}_d$.
Secondly, we prove that when $d\not\in \{-1,-2,-3,-7,-11\}$ all entries in all non-zero frieze patterns over $\mathcal{O}_d$ are rational integers.
\smallskip

Recall that $\mathcal{O}_d=\mathbb{Z}[\omega_d]$ where
\[ \omega_d := \begin{cases}
\sqrt{d} & \text{ if } d \equiv 2,3 \:\:(\md 4), \\
\frac{1+\sqrt{d}}{2} & \text{ if } d \equiv 1 \:\:(\md 4). \\
\end{cases} \]
Figure \ref{fig:quid} contains quiddity cycles of non-zero frieze patterns for $d\in \{-1,-2,-3,-7,-11\}$.
In each of these cases we can thus conclude that $\mathcal{O}_{d}^{\circ}=\mathcal{O}_{d}$. For example, for $d=-11$, with $\omega=\omega_{-11}$ we have $\omega\overline{\omega}=3$ and we obtain the frieze pattern:
$$\begin{array}{cccccccccccc}
 &  & \ddots & &  & &  & & & & & \\
0 & 1 & \omega & 2 & \omega & 1 & 0 & & & & & \\
& 0 & 1 & \overline{\omega} & 2 & \overline{\omega} & 1 & 0 & & & & \\
& & 0 & 1 & \omega & 2 & \omega & 1 & 0 & & &  \\
& & & 0 & 1 & \overline{\omega} & 2 & \overline{\omega} & 1 & 0 & & \\
& & & & 0 & 1 & \omega & 2 & \omega & 1 & 0 & \\
& & & & & 0 & 1 & \overline{\omega} & 2 & \overline{\omega} & 1 & 0 \\
 &  & & & & &  & & &\ddots & & \\
\end{array}
$$
This completes the first part of the proof.
\smallskip

For the second part of the proof we now suppose that $d\not\in \{-1,-2,-3,-7,-11\}$.
Let $\mathcal{C}=(c_{i,j})$ be a non-zero frieze pattern over $\mathcal{O}_d$. We consider the corresponding quiddity cycle $(c_{0,2},\ldots,c_{n+2,n+4})\in (\mathcal{O}_d)^{n+3}$ (where $n$ is the height of $\mathcal{C}$). So we have $\prod_{i=0}^{n+2} \eta(c_{i,i+2}) = -\mathrm{id}$ (see the remark before Lemma \ref{lem:small}). 

We first show that every quiddity cycle $(c_1,\ldots,c_m)$ over $\mathcal{O}_d$ can be reduced to one of the quiddity cycles $(0,0)$, $(1,1,1)$, by applying transformations as in Lemma \ref{lem:small}\,(b).
We proceed by induction on $m$, the length of the quiddity cycle. For $m=2$ and $m=3$ the statement clearly holds since the only quiddity cycle of length 2 is $(0,0)$ 
and the only quiddity cycle of length 3 is $(1,1,1)$ (see \cite[Example 2.7]{CH17}). So let $m\ge 4$. 
By Lemma \ref{lem:small}\,$(a)$ there are two different entries $c_i,c_j$ in the quiddity cycle with $|c_i|<2$ and $|c_j|<2$,
and by Lemma \ref{lem:small}\,$(c)$ these entries can assumed to be non-neighbouring since $m\ge 4$.

The crucial observation is that for $d\not\in \{-1,-2,-3,-7,-11\}$ the only elements in $\mathcal{O}_d$ with absolute value $<2$ are $0$ and $\pm 1$ (this requires some elementary computations which we leave to the reader, and it is not true for $d\in \{-1,-2,-3,-7,-11\}$). This means that in our quiddity cycle $(c_1,\ldots,c_m)$
there are two different non-neighbouring entries equal to $1$, $-1$ or $0$.

If there is a 1 in the quiddity cycle we can remove it by Lemma \ref{lem:small}\,$(b)(i)$ and obtain a quiddity cycle of shorter length $m-1$ to which we apply induction. If there is no 1 in the quiddity cycle, then there must be two non-neighbouring entries equal to $-1$ or $0$. We remove both entries by the transformations in Lemma \ref{lem:small}\,$(b)(ii),(iii)$ and obtain a quiddity cycle of shorter length (note that the signs appearing in these transformations cancel since we apply two transformations involving
a sign) to which we can apply induction. Altogether, inductively we obtain a quiddity cycle of length 2 or 3, that is, we can reduce our initial quiddity cycle $(c_1,\ldots,c_m)$ to $(0,0)$ or $(1,1,1)$.

Now we come back to the quiddity cycle $(c_{0,2},\ldots,c_{n+2,n+4})$ of the frieze pattern $\mathcal{C}$. In particular, the resulting quiddity cycle after the above inductive reduction process has integral entries\footnote{Note that we do not know at this point that the cycles obtained in this process in fact correspond to non-zero frieze patterns because they could a priori contain zeros.}. Reversing this process means that our original quiddity cycle of $\mathcal{C}$ can be obtained from $(0,0)$ or $(1,1,1)$ by transformations in Lemma \ref{lem:small}\,$(b)$; but the new entries resulting from these transformations are still integral (adding/subtracting 1 or adding two entries). 
Therefore, the original quiddity cycle of $\mathcal{C}$ has integral entries and then the rule defining frieze patterns (see Definition \ref{def:friezepattern}) implies that all entries of the frieze pattern $\mathcal{C}$ are rational numbers. However, it is well-known for rings of integers of number fields that $\mathcal{O}_d\cap \mathbb{Q} = \mathbb{Z}$. Thus, all entries of the frieze pattern $\mathcal{C}$ are (non-zero) integers. From Theorem \ref{thm:Fontaine} we can then conclude that $\mathcal{C}$ is a Conway-Coxeter frieze pattern or a twisted Conway-Coxeter pattern and this completes the proof of the theorem. 
\end{proof} 

\begin{Remark}
\begin{enumerate}
\item[{(i)}]
The rings of integers $\mathcal{O}_d$ for $d\in\{-1,-2,-3,-7,-11\}$ are all unique factorization domains (UFDs). This is part of the Baker-Heegner-Stark theorem stating that 
for squarefree $d<0$, the ring of integers $\mathcal{O}_d$ of the imaginary quadratic number field $\mathbb{Q}(\sqrt{d})$ is a UFD if and only if 
$$d\in \{-1,-2,-3,-7,-11,-19,-43,-67,-163\}$$
(thus proving the class number problem for the case of class number 1). 
\item[{(ii)}] Among the imaginary quadratic integers, the cases $d\in\{-1,-2,-3,-7,-11\}$ are precisely the ones where $\mathcal{O}_d$ is a Euclidean ring (with respect to the usual norm function); see \cite[Proposition 4.1]{Lem04}.
\end{enumerate}
\end{Remark}

\end{document}